\documentclass[11pt,a4paper]{article}
\usepackage[T1]{fontenc}
\usepackage{lmodern}
\usepackage{amsmath,amssymb,amsthm,mathtools}
\usepackage[margin=27mm]{geometry}
\usepackage{microtype}
\usepackage{enumitem}
\usepackage{authblk}
\usepackage{cite}
\usepackage[hidelinks,hypertexnames=false]{hyperref}
\usepackage[nameinlink,noabbrev]{cleveref}
\numberwithin{equation}{section}
\newtheorem{theorem}{Theorem}[section]
\newtheorem{lemma}[theorem]{Lemma}
\newtheorem{proposition}[theorem]{Proposition}
\newtheorem{corollary}[theorem]{Corollary}
\theoremstyle{definition}

\theoremstyle{remark}
\newtheorem*{remark}{Remark}
\newcommand{\GraphSys}{\mathbf{G}}
\newcommand{\Ec}{E_{\cap}}
\newcommand{\col}{\operatorname{col}}
\newcommand{\sqcupdot}{\mathbin{\dot\cup}}
\newcommand{\revP}{\overleftarrow{P}}
\setlist[enumerate]{label=(\roman*),itemsep=3pt,topsep=4pt}
\title{Prescribed-Vertex Rainbow Cycle Spectra in Graph Systems under a Local Ore Condition}
\newcommand{\shorttitle}{Prescribed-Vertex Rainbow Cycle Spectra}
\author[1]{Ailian Chen\thanks{Corresponding author: \texttt{elian1425@fzu.edu.cn}.}}
\author[1]{Jiahuang Lin\thanks{Email: \texttt{2182889071@qq.com}.}}
\affil[1]{School of Mathematics and Statistics, Fuzhou University\\
Fuzhou 350108, China}
\date{}
\hypersetup{
 pdftitle={Prescribed-Vertex Rainbow Cycle Spectra in Graph Systems under a Local Ore Condition},
 pdfauthor={Ailian Chen and Jiahuang Lin},
 pdfsubject={Local Ore conditions, rainbow pancyclicity, prescribed-vertex cycle spectra, and extremal graph systems},
 pdfkeywords={graph system, rainbow cycle, pancyclicity, vertex-pancyclicity, local Ore condition, rainbow distance}
}

\begin{document}
\maketitle
\pagestyle{myheadings}
\markboth{A. Chen and J. Lin}{\shorttitle}
\vspace{-1.0em}

\begin{abstract}
Let $\GraphSys=(G_1,\ldots,G_n)$ be a system of graphs on a common
$n$-vertex set, put $\delta(v)=\min_i d_{G_i}(v)$ and
$E_\cap=\bigcap_i E(G_i)$, and write $d^c$ for rainbow distance.
Assume that each $G_i$ is $2$-connected and that
$d^c(u,v)\le2$ and $\delta(u)+\delta(v)<n$ imply $uv\in E_\cap$.
We prove that either every prescribed vertex belongs to a rainbow
cycle of every length $4,\ldots,n$, or all members coincide with the
balanced complete bipartite graph. The non-exceptional systems are
also rainbow pancyclic. Under the stronger local cutoff $4n/3-1$,
every vertex additionally lies on a rainbow triangle, yielding
rainbow vertex-pancyclicity; the cutoff is sharp in the stated
congruence class. The proofs use prescribed-color detours, root-preserving
neighborhood separation, and a rigidity theorem for missing rooted lengths.
\end{abstract}



\section{Introduction}

A recurring theme in Hamiltonian graph theory is that degree conditions
forcing a spanning cycle often yield substantially more information about
cycle lengths. Dirac's and Ore's theorems are classical starting points
\cite{Dirac,Ore}; Bondy's pancyclicity theorem gives an early instance
of this phenomenon~\cite{Bondy,BondyRecent}. Local degree conditions provide
a different perspective: they impose a degree-sum requirement only on
vertex pairs connected by a short path. For ordinary graphs, the results
of Asratian and coauthors include pancyclicity and related
cycle-length conclusions
\cite{HasratianKhachatrian,AsratianBroersmaHeuvelVeldman,
AsratianSarkisian,AsratianHaggkvistSarkisian}.

All graphs in this paper are finite, simple, and undirected.
Here we study these questions in a graph system
$\GraphSys=(G_1,\ldots,G_n)$ on a common vertex set $V$, where $|V|=n$.
The index $i$ serves as a color: an edge assigned color $i$ must lie in
$G_i$, and a path or cycle is \emph{rainbow} when its assigned colors
are distinct. Write
\[
 \delta(v)=\min_{i\in[n]}d_{G_i}(v),\qquad
 E_\cap=\bigcap_{i\in[n]}E(G_i).
\]
The \emph{rainbow distance} $d^c(u,v)$ is the minimum length of a
rainbow $u$--$v$ path, with value $\infty$ when no such path exists.
We say that a system is \emph{rainbow pancyclic} if it contains rainbow
cycles of every length from $3$ to $n$. It is \emph{$[a,b]$-rainbow
vertex-pancyclic} if each vertex lies on a rainbow cycle of every
length $a,\ldots,b$.

The global cross-layer Ore condition takes the form
\begin{equation}\label{eq:global-ore}
 uv\notin E_\cap\quad\Longrightarrow\quad\delta(u)+\delta(v)\ge n.
\end{equation}
Our hypothesis is local: for all distinct $u,v\in V$,
\begin{equation}\tag{LO}\label{eq:LO}
 d^c(u,v)\le2,\quad \delta(u)+\delta(v)<n
 \quad\Longrightarrow\quad uv\in E_\cap.
\end{equation}
Equivalently,
\begin{equation}\tag{$\mathrm{LO}^*$}\label{eq:LOstar}
 d^c(u,v)\le2,\quad uv\notin E_\cap
 \quad\Longrightarrow\quad \delta(u)+\delta(v)\ge n.
\end{equation}
For diagonal systems this is the local Ore condition on nonadjacent
pairs at ordinary distance two.

\paragraph{Related work and the local-to-rooted step.}
Hamiltonicity in graph systems under minimum-degree conditions begins with
Joos and Kim~\cite{JoosKim}; further developments on spanning structures and
cycle spectra include~\cite{LiLiLi2023,LiLiLi,ChengWangZhao,Bradshaw,HuLiLiXu,
LiuChenMa,ChengSunWangWei,YouZhangZhao,YouZhangZhaoBipancyclic,
AharoniMantel,MaCai,Fan,SunWangWeiSurvey}.
Path-spectrum and Hamilton-connected results, while distinct from the problem
here, are given in~\cite{SunWangWeiPanconnected,MaYouZhang,LiLuo,Williamson}.

Three recent theorems play precise roles in our proof.
Chen and Chen~\cite[Theorem~1.4]{ChenChen} derive rainbow Hamiltonicity
from~\eqref{eq:LO} and layerwise $2$-connectivity.
Chen and Zhang~\cite[Theorem~1.1]{ChenZhang} give a
Hamiltonian-to-pancyclic implication under a sum of system degrees;
we use it only to complete the unrooted triangle conclusion.
Li, Wang and Yan~\cite[Theorem~1.6]{LiWangYan} prove the prescribed-vertex
$[4,n]$ spectrum under the \emph{global} cross-layer Ore condition
\eqref{eq:global-ore}; we invoke it only after a rigidity argument has
established that global condition.

The novelty here lies in replacing the global restriction, imposed on
\emph{every} noncommon pair, with~\eqref{eq:LO}, which tests only pairs
of rainbow distance at most two. This is a genuine weakening (see
Proposition~\ref{prop:strict}), and it cannot be made by directly
applying the global theorem: a noncommon pair of rainbow distance
at least three need not satisfy a degree-sum inequality.
The additional work is therefore a root-preserving structural
argument that supplies all intermediate lengths for each fixed vertex.
The above cited theorems are inputs, not new results of this paper.

\begin{theorem}\label{thm:main}
Let $n\ge4$ and let $\GraphSys=(G_1,\ldots,G_n)$ be a graph system on $V$, where $|V|=n$.
Suppose every $G_i$ is $2$-connected and $\GraphSys$ satisfies~\eqref{eq:LO}. Then either
\begin{enumerate}
\item $\GraphSys$ is rainbow pancyclic and $[4,n]$-rainbow vertex-pancyclic; or
\item $n$ is even and $G_i=K_{X,Y}$ for every $i$, for a common balanced bipartition
$V=X\sqcupdot Y$.
\end{enumerate}
\end{theorem}

A stronger local cutoff also guarantees rainbow triangles through every vertex.
Consider
\begin{equation}\tag{$\mathrm{LO}_{4/3}$}\label{eq:LO-43}
 d^c(u,v)\le2,\quad \delta(u)+\delta(v)<\frac{4n}{3}-1
 \quad\Longrightarrow\quad uv\in E_\cap.
\end{equation}
\begin{theorem}\label{thm:local-43}
Let $n\ge3$. If every $G_i$ is $2$-connected and $\GraphSys$ satisfies~\eqref{eq:LO-43},
then $\GraphSys$ is rainbow vertex-pancyclic. The cutoff is sharp in the natural integer sense
for every $n\equiv1\pmod3$.
\end{theorem}

\paragraph{Proof architecture.}
We first prove prescribed-color extensions and a neighborhood separation
principle. Their rooted versions produce the main missing-length rigidity
argument for $4\le m\le n-3$. We treat rooted lengths $4$ and $n-1$
separately, then assemble the full $[4,n]$ spectrum. Finally, we use
the aggregate-degree input to settle existence of an unrooted triangle,
and a separate local argument for rooted triangles at the stronger cutoff.
All uses of external theorems are identified in the statement below and
in the final proof.

\section{Cycle extensions and degree rigidity}\label{sec:rigidity}

\subsection{Preliminaries}\label{sec:prelim}

For a color $a$ and a set $S\subseteq V$, write
\[
 N_a(v,S)=N_{G_a}(v)\cap S,\qquad d_a(v,S)=|N_a(v,S)|.
\]
We abbreviate $N_a(v,V)$ and $d_a(v,V)$ to $N_a(v)$ and $d_a(v)$.
Whenever a rainbow path or cycle is fixed, its edge-color assignment is
fixed as well. Its used colors form the set $\col(H)$; the remaining
colors are called \emph{unused}. All retained edges keep their assigned
colors unless a change is explicitly specified. Length means the number
of edges; a path on $k$ vertices therefore has length $k-1$.

We use the following three input results.

\begin{lemma}[Local Ore counting lemma]\label{lem:local-ore-counting}
Let $xwy$ be a rainbow path of length two with $xy\notin E_\cap$. Then
$\delta(x)+\delta(y)\ge n$. Moreover, for all colors $a,b,c\in[n]$,
\[
 |N_a(x)\cap N_b(y)|\ge
 |N_c(w)\setminus(N_a(x)\cup N_b(y))|.
\]
\end{lemma}
\begin{proof}
The degree-sum assertion follows from~\eqref{eq:LOstar}. For the second assertion, put
$A=N_a(x)$, $B=N_b(y)$ and $C=N_c(w)\setminus(A\cup B)$. Then
$d_a(x)+d_b(y)\ge n$ and $|A\cup B|\le n-|C|$, whence
$|A\cap B|=|A|+|B|-|A\cup B|\ge |C|$.
\end{proof}

\paragraph{Successor-count principle.}
Let $P=x_0\cdots x_m$ be oriented, let $X,Y,W\subseteq V(P)$, and let $E\subseteq V(P)$.
If the successor map sends $W\setminus E$ into $X\setminus Y$, then
\begin{equation}\label{eq:successor-count}
 |X\setminus Y|\ge |W|-|W\cap E|.
\end{equation}
Consequently, if also $|X\setminus Y|\le |W|-q$, then $|W\cap E|\ge q$.
This is just injectivity of the successor map; we invoke it below to isolate the exceptional
vertices in the three recurring neighborhood counts.

\begin{theorem}[Chen and Chen~{\cite[Theorem~1.4]{ChenChen}}]
\label{thm:ham-input}
Let $n\ge3$, and let $\GraphSys=(G_1,\ldots,G_n)$ be a graph system on a common
$n$-vertex set. If every member is $2$-connected and $\GraphSys$ satisfies
\eqref{eq:LO}, then $\GraphSys$ contains a rainbow Hamilton cycle.
\end{theorem}

Chen and Zhang~\cite{ChenZhang} prove the stronger aggregate threshold
$\sum_v\delta(v)\ge\lceil n^2/2\rceil-1$. We require only the following consequence.

\begin{theorem}[Chen and Zhang~{\cite[Theorem~1.1]{ChenZhang}}: aggregate-degree rainbow Bondy theorem]
\label{thm:bondy-input}
Let $\GraphSys=(G_1,\ldots,G_n)$ be a graph system on an $n$-vertex set $V$.
If $\GraphSys$ contains a rainbow Hamilton cycle and
\[
 \sum_{v\in V}\delta(v)\ge\frac{n^2}{2},
\]
then $\GraphSys$ is rainbow pancyclic unless $n$ is even and there is a common
balanced bipartition $V=X\sqcupdot Y$ for which $G_i=K_{X,Y}$ for every
$i\in[n]$.
\end{theorem}

We also use the following consequence of the global Ore theorem of Li, Wang and
Yan~\cite{LiWangYan}. Their theorem is stated in terms of the cross-layer Ore parameter
$\sigma(\GraphSys)$: the condition $\sigma(\GraphSys)\ge n$ requires
$d_{G_p}(u)+d_{G_q}(v)\ge n$ whenever $uv$ fails to be common, uniformly over the colors
$p,q$. Condition~\eqref{eq:global-ore} implies this hypothesis because
$d_{G_p}(u)\ge\delta(u)$ for every $p$.

\begin{theorem}[Consequence of Li, Wang and Yan~{\cite[Theorem~1.6]{LiWangYan}}]\label{thm:global-vertex}
If a graph system $\GraphSys=(G_1,\ldots,G_n)$ satisfies~\eqref{eq:global-ore}, then either
$\GraphSys$ is $[4,n]$-rainbow vertex-pancyclic or $n$ is even and
$G_1=\cdots=G_n=K_{n/2,n/2}$ with a common balanced bipartition.
\end{theorem}

We first record an endpoint-crossing estimate. It is used to turn a degree sum into a prescribed-color detour.

\begin{lemma}[Endpoint crossing]\label{lem:path}
Let $P=z_1\cdots z_k$ be a rainbow path, where $3\le k\le n-1$,
and suppose that $\GraphSys$ has no rainbow $C_k$.
For distinct colors $a,b\notin\col(P)$,
\begin{equation}\label{eq:path-in}
 d_a(z_1,V(P))+d_b(z_k,V(P))\le k-1.
\end{equation}
If $\delta(z_1)+\delta(z_k)\ge n$, then
\begin{equation}\label{eq:path-out}
 d_a(z_1,V\setminus V(P))+d_b(z_k,V\setminus V(P))
 \ge n-k+1,
\end{equation}
and there is a vertex $w\notin V(P)$ with
$z_1w\in E(G_a)$ and $wz_k\in E(G_b)$.
Furthermore, equality in \eqref{eq:path-in} implies that, for every
$j\in[k-1]$, exactly one of
\[
 z_1z_{j+1}\in E(G_a),\qquad z_jz_k\in E(G_b)
\]
holds.
\end{lemma}

\begin{proof}
The edge $z_1z_k$ belongs to neither $G_a$ nor $G_b$; otherwise assigning it color $a$ or $b$
would yield a rainbow $C_k$ on $V(P)$. Define
\[
 I=\{j\in[k-1]:z_1z_{j+1}\in E(G_a)\},\quad
 J=\{j\in[k-1]:z_jz_k\in E(G_b)\}.
\]
If $j\in I\cap J$, deleting $z_jz_{j+1}$ and adding the two
edges in colors $a,b$ gives a rainbow cycle on $V(P)$.
Thus $I\cap J=\varnothing$, which proves \eqref{eq:path-in}.
If equality holds, $I\sqcupdot J=[k-1]$, giving the equality statement.
Subtracting \eqref{eq:path-in} from
$d_a(z_1)+d_b(z_k)\ge\delta(z_1)+\delta(z_k)\ge n$
proves \eqref{eq:path-out}. The two outside neighborhoods have total
size larger than $|V\setminus V(P)|=n-k$, so they intersect.
\end{proof}

\begin{lemma}[A missing shortcut gives a degree sum]\label{lem:shortcut}
Suppose that $\GraphSys$ satisfies \eqref{eq:LO}, that
$H=u_1\cdots u_Lu_1$ is a rainbow cycle with $L\ge4$, and that
there is no rainbow $C_{L-1}$. Then
\[
 \delta(u_i)+\delta(u_{i+2})\ge n
 \qquad(i\in\mathbb Z_L).
\]
\end{lemma}

\begin{proof}
The chord $u_iu_{i+2}$ cannot be common, since it would shorten the rainbow cycle to
$C_{L-1}$. Apply Lemma~\ref{lem:local-ore-counting} to the rainbow path
$u_iu_{i+1}u_{i+2}$.
\end{proof}

\subsection{Extensions with prescribed colors}\label{sec:detours}

Throughout this subsection, assume \eqref{eq:LO}, suppose every member is connected,
and fix a rainbow cycle
\[
 D=x_1x_2\cdots x_mx_1,\qquad 3\le m\le n-3,
\]
while no rainbow $C_{m+1}$ exists. Let
$F=[n]\setminus\col(D)$, so $|F|=n-m\ge3$.
Indices on $D$ are read modulo $m$.

\begin{lemma}[Transfer lemma]\label{lem:transfer}
Let $xy$ be an edge of $D$ with assigned color $e$.
Suppose that $p\notin V(D)$ and $xp\in E(G_a)$ for some $a\in F$.
For each $b\in F\setminus\{a\}$ there is
$q\notin V(D)\cup\{p\}$ such that
\[
 pq\in E(G_e),\qquad qy\in E(G_b).
\]
Consequently, replacing $xy$ by $xpqy$, with colors $a,e,b$ in
that order, gives a rainbow $C_{m+2}$.
\end{lemma}

\begin{proof}
Delete $xy$ from $D$, traverse the resulting $x$--$y$ path, and prepend
$px$ in color $a$.  The resulting rainbow path $P$ has vertex set
$V(D)\cup\{p\}$, hence $m+1$ vertices and endpoints $p,y$, and
\[
 \col(P)=(\col(D)\setminus\{e\})\cup\{a\}.
\]
Thus both $e$ and $b$ are unused on $P$.  Here the two-edge path
$pxy$ is taken in the original graph system, not along $P$: although
$xy$ was omitted from $P$, it remains an edge of $G_e$.  Consequently
$pxy$, using $a$ on $px$ and $e$ on $xy$, is rainbow, so
$d^c(p,y)\le2$.  If $\delta(p)+\delta(y)<n$, then~\eqref{eq:LO}
forces $py\in E_{\cap}$; assigning $py$ color $b$ closes $P$ to a
rainbow $C_{m+1}$, contrary to the standing assumption.  Hence
$\delta(p)+\delta(y)\ge n$.

We may therefore apply Lemma~\ref{lem:path} to the $(m+1)$-vertex path
$P$ with endpoint colors $e,b$; the hypothesis ``no rainbow $C_{m+1}$''
is exactly the standing hypothesis.  The lemma gives a common outside
neighbor $q\notin V(P)=V(D)\cup\{p\}$ such that
$pq\in E(G_e)$ and $qy\in E(G_b)$.  Hence $p,q\notin V(D)$ are
distinct.  Since $a,b\in F$, $a\ne b$, and $e\in\col(D)$, the colors
$a,e,b$ are pairwise distinct, and replacing $xy$ by $xpqy$ gives the
claimed rainbow $C_{m+2}$.
\end{proof}

\begin{lemma}[Prescribed two-vertex extension]\label{lem:detour}
For every edge $xy$ of $D$, with assigned color $e$, and every
ordered pair of distinct colors $a,b\in F$, there are distinct
$p,q\notin V(D)$ such that
\begin{equation}\label{eq:detour}
 xp\in E(G_a),\qquad pq\in E(G_e),\qquad qy\in E(G_b).
\end{equation}
\end{lemma}

\begin{proof}
Call $(x,a)\in V(D)\times F$ \emph{available} if $x$ has a
$G_a$-neighbor outside $D$.  Since $m\le n-3$, the complement of
$V(D)$ is nonempty.  Fix any $a_0\in F$.  Connectivity of $G_{a_0}$
then gives an edge of $G_{a_0}$ between $V(D)$ and its complement, so
at least one pair $(x,a_0)$ is available.

Lemma~\ref{lem:transfer}, applied in either orientation of a cycle edge,
gives the propagation rule
\begin{equation}\label{eq:availability}
 (x,a)\text{ available},\quad xy\in E(D),\quad b\in F\setminus\{a\}
 \quad\Longrightarrow\quad (y,b)\text{ available}.
\end{equation}
We first show that availability at one cycle vertex is independent of
the unused color.  Suppose $(x,a)$ is available and fix
$b\in F\setminus\{a\}$.  Since $|F|\ge3$, choose
$c\in F\setminus\{a,b\}$ and a neighbor $y$ of $x$ on $D$.
Applying~\eqref{eq:availability} to $(x,a)$ with target color $c$
gives $(y,c)$; applying it again along the same edge in the reverse
direction, now with target color $b$, gives $(x,b)$.  Hence $(x,f)$ is
available for every $f\in F$.

Now let $y$ be a neighbor of such an $x$ and let $b\in F$ be arbitrary.
Choose $a\in F\setminus\{b\}$; since $(x,a)$ is available,
\eqref{eq:availability} yields $(y,b)$.  Iterating around the connected
cycle $D$ shows that $(v,f)$ is available for every $v\in V(D)$ and
every $f\in F$.  For the prescribed edge $xy$ and ordered pair
$a,b$, choose a $G_a$-neighbor $p\notin V(D)$ of $x$ and apply
Lemma~\ref{lem:transfer} with the second unused color $b$.
\end{proof}

\begin{corollary}\label{cor:two-step}
For every $i\in\mathbb Z_m$,
\begin{equation}\label{eq:two-step}
 \delta(x_i)+\delta(x_{i+2})\ge n.
\end{equation}
\end{corollary}

\begin{proof}
Choose an edge of $D$ other than the two edges of
$x_ix_{i+1}x_{i+2}$. Such an edge exists even when $m=3$: in that
case it is precisely the third edge $x_{i+2}x_i$. Extend the chosen
edge as in Lemma~\ref{lem:detour}. Since only the chosen edge is replaced, the resulting rainbow $C_{m+2}$
still contains $x_ix_{i+1}x_{i+2}$. As no rainbow $C_{m+1}$ exists,
Lemma~\ref{lem:shortcut} yields~\eqref{eq:two-step}.
\end{proof}

\subsection{Neighborhood separation}\label{sec:separation}

Retain the assumptions and notation of Subsection~\ref{sec:detours}.
The next lemma gives the required separation. All neighborhoods are taken in the whole vertex set
$V$.

\begin{lemma}[Neighborhood separation]\label{lem:separation}
For every edge $xy$ of $D$ and every two distinct colors $a,b\in F$,
\begin{equation}\label{eq:separation}
 N_a(x)\cap N_b(y)=\varnothing.
\end{equation}
\end{lemma}

\begin{proof}
If some $z\notin V(D)$ belonged to $N_a(x)\cap N_b(y)$, then replacing
$xy$ by $xzy$ and assigning colors $a,b$ to the new edges would yield
a rainbow $C_{m+1}$.
Suppose, therefore, that
\begin{equation}\label{eq:internal-z}
 z\in V(D)\setminus\{x,y\},\qquad
 xz\in E(G_a),\qquad zy\in E(G_b).
\end{equation}
Let $s,t$ be the two neighbors of $z$ on $D$, and let $r,h$ be the
assigned colors of $sz,zt$, respectively. Choose the labels $s,t$ so
that the path $D-z$, traversed from $s$ to $t$, encounters $x$ and
then $y$. In particular, $r,h$ are distinct colors of $D$ and neither
is the assigned color $e$ of $xy$.

By Lemma~\ref{lem:detour}, replace $xy$ by $xpqy$ in colors
$a,e,b$, obtaining a rainbow cycle $H$ on $m+2$ vertices.
The vertices $p,q$ lie outside $D$. Write
\[
 Q=H-z,\qquad W=V\setminus V(H).
\]
Then
\[
 V(Q)=(V(D)\setminus\{z\})\cup\{p,q\},\qquad
 V\setminus V(Q)=W\cup\{z\}.
\]
Thus $Q$ is a rainbow $s$--$t$ path on $m+1$ vertices.  Its color set
is obtained from the color set of $H$ by deleting the two colors
$r,h$, so $r,h$ are unused on $Q$.  Moreover
$|W|=n-(m+2)=n-m-2\ge1$; this remains true in the boundary case
$m=n-3$, when $W$ has one vertex.  The edges $sz,zt$ remain consecutive
in $H$, and the standing hypothesis excludes a rainbow $C_{m+1}$.
Hence Lemma~\ref{lem:shortcut} applied to $H$ yields
\begin{equation}\label{eq:st-high}
 \delta(s)+\delta(t)\ge n.
\end{equation}

\medskip
\noindent\emph{Equality on the $s$--$t$ path.}
There is no $w\in W$ with $sw\in E(G_r)$ and $wt\in E(G_h)$.
Indeed, delete the three edges $sz,zt,xy$ from $D$. The vertex
$z$ is then isolated, while the path $D-z$ is cut at the edge $xy$
into two (possibly trivial) path segments, one from $s$ to $x$ and the
other from $y$ to $t$. Add $xz,zy,tw,ws$. Traversing
\[
 s\,D\,x-z-y\,D\,t-w-s
\]
therefore gives a single cycle on $V(D)\cup\{w\}$, of length
$m+1$; this description also covers the boundary case $m=3$, when one
or both of the two retained $D-z$ segments may be trivial. Assign
colors $a,b,h,r$ to $xz,zy,tw,ws$, respectively.
The deleted edges had the distinct colors $r,h,e$, while $a,b$ were
unused on $D$; hence all retained and new colors are pairwise distinct,
and the new cycle is rainbow.
Consequently, the two endpoint neighborhoods outside $Q$, whose
ambient set is $W\cup\{z\}$, intersect exactly in $z$: the vertex $z$
is adjacent to $s$ in color $r$ and to $t$ in color $h$, while the
preceding paragraph excludes every common neighbor in $W$.  Hence
\[
 \begin{aligned}
 d_r(s,W\cup\{z\})+d_h(t,W\cup\{z\})
 &=|N_r(s,W\cup\{z\})\cup N_h(t,W\cup\{z\})|+1\\
 &\le |W|+2.
 \end{aligned}
\]
On the other hand, Lemma~\ref{lem:path} applied to the $(m+1)$-vertex
path $Q$ gives an inside sum at most $m$.  Together with
\eqref{eq:st-high},
\[
 \begin{split}
 n&\le d_r(s)+d_h(t)\\
 &\le m+(|W|+2)=n.
 \end{split}
\]
Thus every inequality is an equality. In particular,
\begin{equation}\label{eq:Q-saturation}
 d_r(s,V(Q))+d_h(t,V(Q))=m.
\end{equation}
Equality outside $Q$ means that the union of the two outside endpoint
neighborhoods is all of $W\cup\{z\}$.  Since their only common vertex
is $z$, every vertex of $W$ belongs to exactly one of them. Therefore
\begin{equation}\label{eq:AB}
 A=N_r(s,W),\qquad B=N_h(t,W)
 \qquad\text{satisfy}\qquad A\sqcupdot B=W.
\end{equation}

\medskip
\noindent\emph{A third unused color.}
Choose $c\in F\setminus\{a,b\}$.  Since $H$ uses the colors of
$D$ together with $a,b$, and $Q$ is obtained by deleting the two
$r,h$-colored edges incident with $z$, the pairs $c,h$ and $r,c$ are
both pairs of unused colors on $Q$.  Apply Lemma~\ref{lem:path} to
$Q$ first with endpoint colors $c,h$ and then with $r,c$; in both
applications the degree lower bound is~\eqref{eq:st-high}.  The lemma
therefore supplies a common endpoint neighbor outside $Q$.  Such a
neighbor cannot lie in $W$, since the construction in the preceding
paragraph would then yield a rainbow $C_{m+1}$, using the four new
colors $c,h,a,b$ or $r,c,a,b$, respectively.  Hence the common
neighbor is $z$ in both applications, and
\begin{equation}\label{eq:third-color}
 sz,zt\in E(G_c).
\end{equation}
For every $w\in W$, we also have
\begin{equation}\label{eq:z-no-W}
 zw\notin E(G_b).
\end{equation}
If $w\in A$, the contrary would insert $w$ into the edge $sz$
of $D$ using colors $r,b$; if $w\in B$, it would insert $w$
into $zt$ using colors $b,h$. Either insertion gives $C_{m+1}$.

At most one of $zp,zq$ belongs to $G_b$. If both belonged, then
$sq\notin E(G_r)$ and $pt\notin E(G_h)$, since either edge together
with one of $zp,zq$ would give a one-vertex insertion into $D$.
Since $pq$ is a consecutive edge of $Q$, the equality statement in
Lemma~\ref{lem:path}, applied to~\eqref{eq:Q-saturation}, requires
exactly one of $sq\in E(G_r)$ and $pt\in E(G_h)$, a contradiction.
Thus
\begin{equation}\label{eq:z-pq}
 |N_b(z,\{p,q\})|\le1.
\end{equation}

\medskip
\noindent\emph{Saturation at an outside vertex.}
Since $W\ne\varnothing$ and $A\sqcupdot B=W$, at least one of
$A,B$ is nonempty. If $A=\varnothing$, reverse the orientation of the
path $D-z$ and simultaneously relabel
\[
(s,t,x,y,r,h,a,b,p,q,A,B)
\longleftrightarrow
(t,s,y,x,h,r,b,a,q,p,B,A).
\]
Here $r,h$ denote the two colors of the edges incident with $z$, not
the prescribed root vertex (which is denoted $r$ only in the later rooted
section). Under this reversal the prescribed detour $x-p-q-y$ becomes
$y-q-p-x$, and the intersection hypothesis becomes
$yz\in E(G_b)$, $zx\in E(G_a)$. In particular, it is essential to
interchange the unused colors $a,b$ as well as the endpoint labels.
The common-neighbor exclusion cycle is read in reverse as
$t\,D\,y-z-x\,D\,s-w-t$. All conclusions
\eqref{eq:st-high}--\eqref{eq:z-pq} transfer with their relabeled
colors and with $A,B$ interchanged. Therefore we may assume
$A\ne\varnothing$ and choose $w\in A$.
Form the path $P$ by taking $ws$ in color $r$ and then traversing
the $s$--$z$ path in $D$ that avoids $sz$.  This longer $s$--$z$
arc contains every vertex of $D$ and every edge of $D$ except $sz$;
after adding $ws$ in the released color $r$, the path $P$ has vertex
set $V(D)\cup\{w\}$, has $m+1$ vertices, endpoints $w,z$, and uses
exactly the $m$ colors of $D$.  In particular $a,b$ are unused on
$P$.  The path $wsz$ in colors $r,c$ is rainbow by
\eqref{eq:third-color}. If $\delta(w)+\delta(z)<n$, then~\eqref{eq:LO} gives $wz\in E_{\cap}$, and assigning
$wz$ color $a$ yields a rainbow $C_{m+1}$, a contradiction.
Thus $\delta(w)+\delta(z)\ge n$.
Apply Lemma~\ref{lem:path} to $P$ with endpoint colors $a,b$.
The vertices outside $P$ form the set
\[
 U=(W\setminus\{w\})\cup\{p,q\},\qquad |U|=n-m-1.
\]
By \eqref{eq:z-no-W} and \eqref{eq:z-pq}, $d_b(z,U)\le1$,
whereas \eqref{eq:path-out} gives
\[
 d_a(w,U)+d_b(z,U)\ge n-m=|U|+1.
\]
Since always $d_a(w,U)\le|U|$, both inequalities must be tight:
$d_a(w,U)=|U|$ and $d_b(z,U)=1$.  Moreover
\eqref{eq:z-no-W} shows that the unique $G_b$-neighbor of $z$ in $U$
comes from $\{p,q\}$. Thus
\begin{equation}\label{eq:w-full}
 N_a(w,U)=U,\qquad N_b(z,U)=\{p_0\}
\end{equation}
for one $p_0\in\{p,q\}$. Write
$\{p_0,q_0\}=\{p,q\}$.

If $B\ne\varnothing$, choose $v\in B$. The edges
$sw,wv,vt$ have colors $r,a,h$, by \eqref{eq:AB} and
\eqref{eq:w-full}. Replacing $szt$ in $D$ by $swvt$ gives a
rainbow $C_{m+1}$, a contradiction. Hence $B=\varnothing$.

Now take the path $R$ that starts with $p_0z$ in color $b$ and
continues along the $z$--$t$ path of $D$ avoiding $zt$.  The long
$z$--$t$ arc uses every color of $D$ except $h$, so
\[
 \col(R)=(\col(D)\setminus\{h\})\cup\{b\}.
\]
Thus $R$ has $m+1$ vertices and the colors $h,a$ are unused on it.
The path $p_0zt$ in colors $b,h$ shows that
$d^c(p_0,t)\le2$. If $\delta(p_0)+\delta(t)<n$, then~\eqref{eq:LO} gives
$p_0t\in E_{\cap}$, and assigning this edge color $a$ closes $R$ to a rainbow
$C_{m+1}$. Hence $\delta(p_0)+\delta(t)\ge n$.
Lemma~\ref{lem:path}, with color $a$ at $p_0$ and color $h$ at
$t$, yields a common neighbor outside $R$. Since
\[
 V\setminus V(R)=W\cup\{q_0\},\qquad N_h(t,W)=B=\varnothing,
\]
that common neighbor must be $q_0$, and in particular
$q_0t\in E(G_h)$. Equation~\eqref{eq:w-full} also gives
$wq_0\in E(G_a)$. Replacing $szt$ in $D$ by $swq_0t$, with
colors $r,a,h$, again yields a rainbow $C_{m+1}$.
This contradiction excludes~\eqref{eq:internal-z}.
\end{proof}

\section{Rooted cycle rigidity and boundary lengths}\label{sec:rooted}

Throughout this section, a cycle is \emph{$r$-rooted} if it contains a prescribed vertex
$r\in V$.  The extension machinery from Section~\ref{sec:rigidity} remains available; the only
additional issue is whether a shortening deletes $r$.

\subsection{Root-preserving extensions}\label{sec:rooted-extensions}

\begin{lemma}[Rooted endpoint crossing]\label{lem:rooted-path}
Let $P=z_1\cdots z_k$ be a rainbow path containing $r$, where $3\le k\le n-1$, and suppose
there is no $r$-rooted rainbow $C_k$ on $V(P)$. For distinct colors
$a,b\notin\col(P)$,
\[
 d_a(z_1,V(P))+d_b(z_k,V(P))\le k-1.
\]
If $\delta(z_1)+\delta(z_k)\ge n$, then
\[
 d_a(z_1,V\setminus V(P))+d_b(z_k,V\setminus V(P))\ge n-k+1,
\]
and the two outside endpoint neighborhoods intersect.
If equality holds in the first display, then for every $j\in[k-1]$ exactly one of
$z_1z_{j+1}\in E(G_a)$ and $z_jz_k\in E(G_b)$ holds.
\end{lemma}

\begin{proof}
For each $j\in[k-1]$, the two edges
\[
 z_1z_{j+1}\in E(G_a),\qquad z_jz_k\in E(G_b)
\]
cannot both be present. Otherwise, traverse
\[
 z_1Pz_j\;z_k\revP z_{j+1}\;z_1.
\]
This is a cycle on precisely $V(P)$, even when $j=1$ or $j=k-1$.
The retained edges of $P$ have their original colors, while the two displayed
crossing edges receive the unused colors $b$ and $a$, respectively.
Thus the cycle is rainbow and contains $r$, a contradiction.
The $(k-1)$ pairs of prospective crossing edges account for every
$G_a$-neighbor of $z_1$ on $P$ and every $G_b$-neighbor of $z_k$
on $P$, including the boundary cases. Hence their total number is at most $k-1$.
If equality holds, exactly one edge is present in each pair.

Finally, the degree hypothesis gives
\[
\begin{aligned}
 d_a(z_1,V\setminus V(P))+d_b(z_k,V\setminus V(P))
 &\ge\delta(z_1)+\delta(z_k)
      -d_a(z_1,V(P))-d_b(z_k,V(P))\\
 &\ge n-k+1.
\end{aligned}
\]
There are only $n-k$ vertices outside $P$, so the two outside
neighborhoods have a common vertex.
\end{proof}

\begin{lemma}[Root-preserving shortcut]\label{lem:rooted-shortcut}
Let $H=u_1\cdots u_Lu_1$ be an $r$-rooted rainbow cycle with $L\ge4$, and suppose that no
$r$-rooted rainbow $C_{L-1}$ exists. If $u_{i+1}\ne r$, then
\[
 \delta(u_i)+\delta(u_{i+2})\ge n.
\]
\end{lemma}

\begin{proof}
The two consecutive edges $u_iu_{i+1},u_{i+1}u_{i+2}$ form a rainbow $2$-path. If the displayed
sum were smaller than $n$, condition~\eqref{eq:LO} would make $u_iu_{i+2}$ common. Replacing
$u_iu_{i+1}u_{i+2}$ by this chord, in either released color, deletes only the non-root vertex
$u_{i+1}$ and therefore gives an $r$-rooted rainbow $C_{L-1}$, a contradiction.
\end{proof}

Fix now an $r$-rooted rainbow cycle
\[
 D=x_0x_1\cdots x_{m-1}x_0,\qquad x_0=r,
\]
where $4\le m\le n-3$, and assume that no $r$-rooted rainbow $C_{m+1}$ exists. Put
$F=[n]\setminus\col(D)$, so $|F|\ge3$.

\begin{lemma}[Rooted prescribed extension]\label{lem:rooted-detour}
For every edge $xy$ of $D$, with assigned color $e$, and every ordered pair of distinct colors
$a,b\in F$, there are distinct $p,q\notin V(D)$ such that
\[
 xp\in E(G_a),\qquad pq\in E(G_e),\qquad qy\in E(G_b).
\]
Thus replacing $xy$ by $xpqy$ gives an $r$-rooted rainbow $C_{m+2}$.
\end{lemma}

\begin{proof}
We first establish the rooted transfer rule directly.  Call $(x,a)\in V(D)\times F$
\emph{available} if $x$ has a $G_a$-neighbor outside $D$. Suppose $(x,a)$ is available, let
$xy$ be a cycle edge of color $e$, choose $b\in F\setminus\{a\}$, and take
$p\notin V(D)$ with $xp\in E(G_a)$. Delete $xy$, traverse the resulting $x$--$y$ path in
$D$, and prepend $px$ in color $a$. The resulting rainbow path $P$ has vertex set
$V(D)\cup\{p\}$, hence contains $r$, and
\[
 \col(P)=(\col(D)\setminus\{e\})\cup\{a\};
\]
thus $e$ and $b$ are unused on $P$.  The two-edge path $pxy$ in colors $a,e$ shows that
$d^c(p,y)\le2$. If $\delta(p)+\delta(y)<n$, then~\eqref{eq:LO} makes $py$ common, and
assigning it color $b$ closes $P$ to an $r$-rooted rainbow $C_{m+1}$, contrary to the
standing hypothesis. Hence $\delta(p)+\delta(y)\ge n$. Lemma~\ref{lem:rooted-path}, applied
to $P$ with endpoint colors $e,b$, therefore gives
$q\notin V(D)\cup\{p\}$ with
\[
 pq\in E(G_e),\qquad qy\in E(G_b).
\]
In particular, $(y,b)$ is available. We have proved
\[
 (x,a)\text{ available},\quad xy\in E(D),\quad b\in F\setminus\{a\}
 \quad\Longrightarrow\quad (y,b)\text{ available}.
\]

At least one pair is available because $G_f$ is connected for every $f\in F$ and
$V(D)$ is a nonempty proper subset of $V$.  Starting from any available $(x,a)$, let
$b\in F\setminus\{a\}$ and choose $c\in F\setminus\{a,b\}$. Applying the transfer rule
along one cycle edge and then back along the same edge shows that $(x,b)$ is available.
Thus every unused color is available at $x$; iterating the transfer rule around $D$ shows
that every pair $(v,f)\in V(D)\times F$ is available. For the prescribed edge $xy$ and
ordered colors $a,b$, choose $p\notin V(D)$ with $xp\in E(G_a)$ and apply the rooted
transfer rule once more. This gives the required $q$ and the stated rooted $C_{m+2}$.
\end{proof}

\begin{lemma}[Rooted neighborhood separation]\label{lem:rooted-separation}
For every edge $xy\in E(D)$ and distinct $a,b\in F$,
\[
 N_a(x)\cap N_b(y)\subseteq\{r\}.
\]
In particular, if $xy$ is incident with $r$, then the intersection is empty.
\end{lemma}

\begin{proof}
An intersection vertex outside $D$ immediately inserts into the edge $xy$ and produces an
$r$-rooted $C_{m+1}$. Suppose $z\in V(D)\setminus\{x,y,r\}$ lies in the intersection.
We apply the argument of Lemma~\ref{lem:separation} with the following
root-preservation checks. The prescribed detour $H$ is provided by
Lemma~\ref{lem:rooted-detour}, so $V(D)\subseteq V(H)$ and $r\in V(H)$.
The only missing-shortcut application on $H$ deletes $z$ (the
middle vertex of $szt$), which is not $r$; hence
Lemma~\ref{lem:rooted-shortcut} supplies
$\delta(s)+\delta(t)\ge n$. The path $Q=H-z$ contains $r$,
so all three endpoint-crossing applications on $Q$ are valid with
Lemma~\ref{lem:rooted-path}. The explicitly constructed
$C_{m+1}$ using $sDx-z-yDt-w-s$ contains $V(D)$, and every
one-vertex insertion into an edge of $D$ also contains $V(D)$.
The auxiliary paths $P$ and $R$ each contain the full vertex set of
$D$, so their endpoint-crossing applications are rooted as well.
In the $A=\varnothing$ branch, the reversal explicitly interchanges
$(x,y)$, $(a,b)$, $(p,q)$ and $(s,t)$ together with the corresponding
incident-edge colors; it leaves the vertex set of $D$ and hence the
prescribed root fixed. This is not a reversal with $a,b$ held fixed.
Finally, the two cycles formed by replacing $szt$ with $swvt$ or
$swq_0t$ delete only $z$ from $D$; they therefore retain $r$.
Thus every cycle used to contradict the standing hypothesis still contains the
root, and every endpoint-crossing count is applied to a path containing $r$.
In particular, the required color ledger is as follows: when $xy$ of
color $e$ is replaced by $xpqy$, the added edges have colors $a,e,b$;
in the final branch of the separation argument, $szt$ (colors $r_0,h_0$
in the unrooted proof's notation) is replaced by $swvt$ or $swq_0t$
(colors $r_0,a,h_0$), with $a$ unused on $D$. Here $r_0,h_0$
stand for the two original colors at $z$, not for the prescribed
root vertex. No retained edge receives either of these released
colors. Consequently, each color choice, neighborhood separation,
and equality case in the proof of Lemma~\ref{lem:separation} remains
valid with ``no $r$-rooted $C_{m+1}$'' in place of ``no $C_{m+1}$.''
This excludes each intersection vertex $z\ne r$.
If $xy$ is incident with $r$, the only possible intersection point from the first assertion is
$r$ itself, but a simple graph has no loop at $r$.
\end{proof}

Two consequences will be used repeatedly.  First,
if $x_{i+1}\ne r$, then
\begin{equation}\label{eq:rooted-two-step-safe}
 \delta(x_i)+\delta(x_{i+2})\ge n.
\end{equation}
Indeed, extend an edge of $D$ different from the two edges of
$x_ix_{i+1}x_{i+2}$ by Lemma~\ref{lem:rooted-detour}. The resulting rooted $C_{m+2}$ still contains
the consecutive segment $x_ix_{i+1}x_{i+2}$, whose middle vertex is not $r$; applying
Lemma~\ref{lem:rooted-shortcut} gives the displayed inequality.
Second, for every $c\in F$,
\begin{equation}\label{eq:root-second-exclusion}
 rx_2,\;rx_{m-2}\notin E(G_c).
\end{equation}
For example, if $rx_2\in E(G_c)$, choose two other unused colors and extend an edge different
from $rx_1,x_1x_2$ to an $r$-rooted $C_{m+2}$; then replace $rx_1x_2$ by $rx_2$ in color $c$.
The other assertion is symmetric.

\subsection{Rooted missing-length rigidity}\label{sec:rooted-rigidity-proof}

The following elementary switching lemma makes explicit a recoloring used
in the maximal-root argument. In particular, the endpoint cases
$i=1$ and $i=m-2$ are included; in these cases a root edge is
reused with a new color, rather than duplicated.

\begin{lemma}[Root pivot at an interior edge]\label{lem:root-pivot}
Let
\[
 D=rx_1x_2\cdots x_{m-1}r,\qquad m\ge4,
\]
be a rainbow $m$-cycle, let $e_i$ be the assigned color of
$x_ix_{i+1}$, and suppose that $x_1x_{m-1}\in E_\cap$.
Fix $1\le i\le m-2$, and let $a,b$ be two distinct colors unused
on $D$. If
\[
 rx_i\in E(G_a),\qquad rx_{i+1}\in E(G_b),
\]
then there exists an $r$-rooted rainbow $m$-cycle with root
neighbors $x_i,x_{i+1}$.
\end{lemma}

\begin{proof}
Traverse the cycle in the order
\begin{equation}\label{eq:root-pivot-order}
 r,\ x_i,x_{i-1},\ldots,x_1,\ x_{m-1},x_{m-2},\ldots,x_{i+1},\ r.
\end{equation}
The two displayed subpaths are interpreted as one-vertex paths if
$i=1$ or $i=m-2$, respectively. Thus the order lists each vertex
of $D$ exactly once before returning to $r$. Retain all original
edge colors within the two subpaths. Color the chord
$x_1x_{m-1}$ with $e_i$, and the two new root edges
$rx_i,rx_{i+1}$ with $a,b$, respectively. The retained edges
are precisely the edges of $D$ other than
$rx_1$, $x_ix_{i+1}$, and $x_{m-1}r$, so their colors
exclude the three distinct colors assigned to those edges.
The three added colors $e_i,a,b$ are pairwise distinct and
avoid all retained colors. If $i=1$, the edge $rx_i=rx_1$
is an old root edge with a new color $a$; if $i=m-2$,
the same observation applies to $rx_{i+1}=rx_{m-1}$.
Consequently there is no repeated edge or vertex, and the
resulting cycle is rainbow with the asserted root neighbors.
This argument uses the common chord only in the released original color $e_i$;
it does not assume that an arbitrary layer edge can take an arbitrary color.
\end{proof}

\begin{theorem}[Rooted missing-length rigidity]\label{thm:rooted-rigidity}
Let $\GraphSys=(G_1,\ldots,G_n)$ satisfy~\eqref{eq:LO}, suppose every member is connected, and
fix $r\in V$.  Let $4\le m\le n-3$.  If $\GraphSys$ contains an $r$-rooted rainbow $C_m$ but no
$r$-rooted rainbow $C_{m+1}$, then $n$ is even and
\[
 G_1=\cdots=G_n=K_{n/2,n/2}
\]
with a common balanced bipartition.
\end{theorem}

\begin{proof}[Proof of Theorem~\ref{thm:rooted-rigidity}]
The proof consists of five stages: a maximal-root degree-sum bound,
a propagation of half-degree equalities, a common unused-color
neighborhood partition, a transfer to vertices outside the cycle,
and an upgrade to the global Ore condition. All detours retain the root;
$|F|=n-m\ge3$ throughout, including $m=n-3$.

\smallskip\noindent\textit{(R1) Maximal root-neighbor degree sum.}
Choose, among all $r$-rooted rainbow $m$-cycles,
\[
 D=x_0x_1\cdots x_{m-1}x_0,\qquad x_0=r,
\]
so that
\[
 \rho(D)=\delta(x_1)+\delta(x_{m-1})
\]
is maximum.  We claim
\begin{equation}\label{eq:rho-lower}
 \rho(D)\ge n.
\end{equation}
Suppose instead that $\rho(D)<n$. The rainbow two-edge path
$x_{m-1}rx_1$ and~\eqref{eq:LO} give $x_{m-1}x_1\in E_\cap$.
Fix distinct colors $a,b\in F$. For every edge $x_ix_{i+1}$ of $D$,
Lemma~\ref{lem:rooted-separation} gives
\[
 N_a(x_i)\cap N_b(x_{i+1})\subseteq\{r\}.
\]
If this intersection is empty, then
$\delta(x_i)+\delta(x_{i+1})\le d_a(x_i)+d_b(x_{i+1})\le n$.
For either edge incident with $r$, the intersection is necessarily empty,
since a simple graph has no loop at $r$. If the intersection is nonempty,
then $x_ix_{i+1}$ is a non-root edge, say with $1\le i\le m-2$, and
$rx_i\in E(G_a)$, $rx_{i+1}\in E(G_b)$. Lemma~\ref{lem:root-pivot}
applies because $x_1x_{m-1}$ is common; its explicit cycle order
\eqref{eq:root-pivot-order} includes the endpoint cases $i=1,m-2$.
By maximality of $\rho(D)$, the resulting cycle gives
$\delta(x_i)+\delta(x_{i+1})\le\rho(D)<n$. Thus in every case
\begin{equation}\label{eq:all-adj-upper}
 \delta(x_i)+\delta(x_{i+1})\le n\qquad(i\in\mathbb Z_m).
\end{equation}
Together with~\eqref{eq:rooted-two-step-safe}, these inequalities imply
\[
 \delta(x_1)\le\delta(x_2)\le\cdots\le\delta(x_{m-1})\le\delta(x_0).
\]
Indeed, comparing the safe two-step inequality at
$x_{i-1}x_ix_{i+1}$ with the adjacent upper bound on
$x_{i-1}x_i$ gives $\delta(x_i)\le\delta(x_{i+1})$
for $1\le i\le m-2$. Comparing the two-step inequality
at $x_{m-2}x_{m-1}x_0$ with the upper bound on
$x_{m-2}x_{m-1}$ gives
$\delta(x_{m-1})\le\delta(x_0)$.
All these middle vertices are different from the root, including
when $m=4$.
Moreover,
$\delta(x_0)+\delta(x_2)\ge n$ and
$\delta(x_{m-1})+\delta(x_0)\le n$, so
$\delta(x_2)\ge\delta(x_{m-1})$. Thus
$\delta(x_2)=\cdots=\delta(x_{m-1})=n-\delta(x_0)$.
Finally $\delta(x_1)+\delta(x_3)\ge n$ gives
$\delta(x_1)\ge\delta(x_0)$, whence
$\rho(D)\ge n$, a contradiction. This proves~\eqref{eq:rho-lower}.

\smallskip\noindent\textit{(R2) Equality propagation at the root.}
By~\eqref{eq:root-second-exclusion}, Lemma~\ref{lem:rooted-separation} is strict on the edges
$x_1x_2$ and $x_{m-2}x_{m-1}$, as well as on the two root edges. Therefore
\[
\begin{aligned}
 \delta(x_0)+\delta(x_1)&\le n,&
 \delta(x_1)+\delta(x_2)&\le n,\\
 \delta(x_{m-2})+\delta(x_{m-1})&\le n,&
 \delta(x_{m-1})+\delta(x_0)&\le n.
\end{aligned}
\]
We spell out the resulting equalities.  From~\eqref{eq:rho-lower} and
$\delta(x_0)+\delta(x_1)\le n$ we obtain
$\delta(x_{m-1})\ge\delta(x_0)$.  On the other hand,
$\delta(x_{m-2})+\delta(x_0)\ge n$ and
$\delta(x_{m-2})+\delta(x_{m-1})\le n$ give the reverse inequality.  Hence
$\delta(x_{m-1})=\delta(x_0)$, and then
$\delta(x_0)+\delta(x_1)=n$.  Moreover,
$\delta(x_0)+\delta(x_2)\ge n$ and
$\delta(x_1)+\delta(x_2)\le n$ yield
$\delta(x_0)\ge\delta(x_1)$.  Since
$2\delta(x_0)=\delta(x_0)+\delta(x_{m-1})\le n$ while
$\delta(x_0)+\delta(x_1)=n$, it follows that
$\delta(x_0)=\delta(x_1)=\delta(x_{m-1})=n/2$.  The two remaining displayed
upper bounds, together with the corresponding safe two-step inequalities, now give
\begin{equation}\label{eq:root-local-half}
 \delta(x_0)=\delta(x_1)=\delta(x_2)=\delta(x_{m-2})=\delta(x_{m-1})=\frac n2.
\end{equation}
The same strict-separation argument on $x_2x_3$ gives
$\delta(x_3)=n/2$: the only possible common neighbor left by
Lemma~\ref{lem:rooted-separation} is the root $x_0$, but
$x_0x_2$ is absent in every unused color by
\eqref{eq:root-second-exclusion}. Combine the resulting adjacent
upper bound with the safe two-step lower bound on $x_1x_2x_3$.
For $m=4$ the conclusion was already included in
\eqref{eq:root-local-half}. Thus no nonexistent $x_4$ or further
internal edge is required when $m=4$; likewise, the detour arguments
above require only the three unused colors still available when
$m=n-3$.

\smallskip\noindent\textit{(R3) Unused-color neighborhood partition.}
Now fix three distinct colors $a,b,c\in F$. Strict separation on the root edge $x_0x_1$,
together with~\eqref{eq:root-local-half}, gives
\[
 N_a(x_0)\sqcupdot N_b(x_1)=V.
\]
Replacing $a$ by $c$ while keeping $b$ fixed shows
$N_a(x_0)=N_c(x_0)=V\setminus N_b(x_1)$; similarly, keeping $a$ fixed and varying the second
color shows that $N_b(x_1)$ is independent of $b\in F$. Since $|F|\ge3$, the same comparison
covers every unused color. Hence there is a partition
\[
 V=A\sqcupdot B,\qquad |A|=|B|=n/2,
\]
such that
\[
 N_f(x_0)=A,\qquad N_f(x_1)=B\qquad(f\in F).
\]
Next,~\eqref{eq:root-second-exclusion} rules out the possible root overlap on $x_1x_2$.
Thus strict separation on $x_1x_2$, together with $N_f(x_1)=B$ and
$\delta(x_2)=n/2$, gives $N_f(x_2)=A$ for every $f\in F$. The same exclusion rules out the
root overlap on $x_2x_3$, and strict separation there, using $\delta(x_3)=n/2$, gives
$N_f(x_3)=B$. Consequently
\begin{equation}\label{eq:root-partition}
 N_f(x_0)=A,\quad N_f(x_1)=B,\quad N_f(x_2)=A,\quad N_f(x_3)=B
 \qquad(f\in F).
\end{equation}

\smallskip\noindent\textit{(R4) Degrees of vertices outside the rooted cycle.}
Let $w\notin V(D)$. We verify the degree conclusion for both classes of
\eqref{eq:root-partition}. Choose two distinct colors $a,b\in F$.

If $w\in A$, the edges $x_0w,x_2w$ belong to both $G_a$ and $G_b$.
Replace $x_0x_1x_2$ in $D$ by $x_0wx_2$, assigning $a,b$ to the two
new edges, respectively. This gives an $r$-rooted rainbow $m$-cycle
$D_w$; the two removed edge colors become unused, so $D_w$ again has
at least three unused colors. Its consecutive segment $wx_2x_3$
has middle vertex $x_2\ne r$. The rooted two-step inequality, applied
to $D_w$, therefore gives
\[
 \delta(w)+\delta(x_3)\ge n,
 \qquad\text{hence}\qquad \delta(w)\ge n/2.
\]
On the other hand, the rooted neighborhood-separation lemma on the
root edge $x_0w$ of $D_w$ gives, for any two distinct unused colors
$f,g$ of $D_w$,
\[
 N_f(x_0)\cap N_g(w)=\varnothing.
\]
Thus $\delta(x_0)+\delta(w)\le d_f(x_0)+d_g(w)\le n$,
so $\delta(w)\le n/2$.

If $w\in B$, replace the segment $x_1x_2x_3$ by $x_1wx_3$,
again using $a,b$. The resulting rooted rainbow $m$-cycle $D'_w$
has $w$ at position two and at least three unused colors. Its
segment $x_0x_1w$ has non-root middle vertex $x_1$, so
\[
 \delta(x_0)+\delta(w)\ge n,
 \qquad\text{and hence}\qquad \delta(w)\ge n/2.
\]
By the root-second-exclusion property applied to $D'_w$, the root
$x_0$ is not adjacent to its second successor $w$ in any unused
color of $D'_w$. Consequently, the possible root overlap in rooted
neighborhood separation on the edge $x_1w$ is absent. For distinct
unused colors $f,g$ of $D'_w$,
\[
 N_f(x_1)\cap N_g(w)=\varnothing,
 \quad\text{so}\quad
 \delta(x_1)+\delta(w)\le n.
\]
Since $\delta(x_0)=\delta(x_1)=\delta(x_3)=n/2$ by
\eqref{eq:root-local-half} and the immediately preceding paragraph,
both cases give
\begin{equation}\label{eq:outside-half}
 \delta(w)=n/2\qquad(w\notin V(D)).
\end{equation}

\smallskip\noindent\textit{(R5) Global Ore and extremal classification.}
It remains only to rule out a cycle vertex of smaller system degree. For an edge $x_ix_{i+1}$,
apply Lemma~\ref{lem:rooted-detour} and write the new segment as
$x_i-p-q-x_{i+1}$, where $p,q\notin V(D)$. By~\eqref{eq:outside-half}, $\delta(q)=n/2$.
If $\delta(x_i)<n/2$, then the rainbow path $x_i p q$ and~\eqref{eq:LO} make $x_iq$ common;
shortening through $p$ gives an $r$-rooted rainbow $C_{m+1}$:
replace the two-edge segment $x_i-p-q$, whose assigned colors are
$a,e$, by $x_i-q$ in the released color $a$. The other inserted edge
$q x_{i+1}$ retains its unused color $b$, while $e$ was the color
of the original deleted edge $x_ix_{i+1}$. Hence no color is duplicated
and the resulting cycle still contains $r$. This is a contradiction. Hence
\[
 \delta(v)\ge n/2\qquad(v\in V).
\]
Consequently~\eqref{eq:global-ore} holds. Theorem~\ref{thm:global-vertex} now implies that either
$\GraphSys$ is $[4,n]$-rainbow vertex-pancyclic or it is the common balanced complete bipartite
system. The first alternative contradicts the assumed absence of an $r$-rooted $C_{m+1}$.
\end{proof}

\subsection{The rooted 4-cycle}\label{sec:rooted-c4}

\begin{lemma}[Rooted $C_4$]\label{lem:rooted-c4}
Assume every $G_i$ is $2$-connected and $\GraphSys$ satisfies~\eqref{eq:LO}. Then every
$r\in V$ lies on a rainbow $C_4$.
\end{lemma}

\begin{proof}
For $n=4$ this is Theorem~\ref{thm:ham-input}. Assume $n\ge5$ and fix $r$.
Form a bipartite graph $B_r$ with left part $[n]$ and right part $V\setminus\{r\}$, joining
$i$ to $v$ when $rv\in E(G_i)$. Since every $G_i$ is $2$-connected, every left vertex has
degree at least two.

Assume that $r$ lies on no rainbow $C_4$.  If
$rx\in E(G_a)$ and $ry\in E(G_b)$ with $a\ne b$ and $x\ne y$, then among the colors in
$[n]\setminus\{a,b\}$ at most one can miss the edge $xy$. Indeed, if distinct colors $c,d$
miss $xy$, then $xry$ is a rainbow two-edge path and~\eqref{eq:LO} gives
$\delta(x)+\delta(y)\ge n$. Hence
$d_c(x)+d_d(y)\ge n$. Since neither $G_c$ nor $G_d$ contains $xy$,
$N_c(x)\cup N_d(y)\subseteq V\setminus\{x,y\}$, so the two neighborhoods have at least two
common vertices. One of them may be $r$, but another vertex $z\ne r$ yields the rainbow cycle
$r x z y r$ in colors $a,c,d,b$, a contradiction.

The matching number of $B_r$ is at most two.  Otherwise choose a matching
$(a,x),(b,y),(c,z)$. If $n\ge6$, put $S=[n]\setminus\{a,b,c\}$. The preceding paragraph says
that each of $xy$ and $yz$ is present in all but at most one color of $S$; since $|S|\ge3$,
we may choose distinct $p,q\in S$ with $xy\in E(G_p)$ and $yz\in E(G_q)$. Then
$rxyzr$ is a rainbow $C_4$. If $n=5$, let $S=\{p,q\}$. Each of the three pairs $xy,yz,zx$ is present in at least one of
these two colors. If two pairs can be assigned different colors from $S$, then the corresponding
two-edge path through $x,y,z$, together with the root edges at its two ends, gives a rainbow
$C_4$. Otherwise each pair has exactly one available color, and all three pairs have the same
available color: if one pair were available in both colors, it could be assigned the color opposite
to that of any other pair; and if two singleton available-color sets differed, those two pairs could
be assigned different colors. After interchanging $p,q$ if necessary, all three pairs are therefore
absent from $G_q$. Then $x,y,z$ are pairwise nonadjacent in $G_q$, so each has $G_q$-degree at
most two. Yet every pair among $x,y,z$ is joined by a rainbow two-edge path through $r$, and
~\eqref{eq:LO} gives system-degree, hence $G_q$-degree, sum at least five for each such pair, a
contradiction. Thus again a rainbow $C_4$ exists.

By K\H{o}nig's theorem, $B_r$ has a vertex cover of size at most two. Such a cover cannot contain
a left vertex.  A single left vertex misses every edge incident with any other color; two left
vertices miss the edges incident with a third color; and a cover consisting of one left and one
right vertex would force every other left vertex to have degree at most one.  A cover consisting
of a single right vertex is also impossible because every left vertex has degree at least two.
Consequently a minimum cover has exactly two right vertices $p,q\in V\setminus\{r\}$.
Since each color has at least two neighbors of $r$,
\[
 N_{G_i}(r)=\{p,q\}\qquad(i\in[n]),
\]
so $rp,rq$ are common and $\delta(r)=2$.

Fix a color $k$. Because $G_k-r$ is connected, one of $p,q$, say $p$, has a $G_k$-neighbor
$v\notin\{r,p,q\}$. Color the common edge $rp$ with a color different from $k$; then $rpv$ is a
rainbow two-edge path. Since $rv$ belongs to no member,~\eqref{eq:LO} gives
$\delta(v)\ge n-2$. Thus in every $G_j$ the vertex $v$ is adjacent to every vertex except $r$;
in particular $vp,vq$ are common. The common cycle $rpvqr$ can be assigned four distinct
colors, the final contradiction.
\end{proof}

\subsection{The rooted \texorpdfstring{$(n-1)$}{(n-1)}-cycle}\label{sec:rooted-cn1}

\begin{lemma}[Rooted $C_{n-1}$]\label{lem:rooted-cn1}
Let $n\ge6$. Suppose every $G_i$ is $2$-connected and $\GraphSys$ satisfies~\eqref{eq:LO}.
For every $r\in V$, either $r$ lies on a rainbow $C_{n-1}$ or
$G_1=\cdots=G_n=K_{n/2,n/2}$ with a common balanced bipartition.
\end{lemma}

\begin{proof}
The five claims below form a dependency chain: C1 provides common
Hamilton edges, C2 excludes root-safe distance-two chords, C3 forces
common distance-three chords, C4 provides a two-color degree bound,
and C5 identifies the balanced bipartite obstruction. Every intermediate
rainbow cycle used for a contradiction retains $x_1$.

Fix $r=x_1$ and, by Theorem~\ref{thm:ham-input}, a rainbow Hamilton cycle
\[
 C=x_1x_2\cdots x_nx_1,
\]
with $x_ix_{i+1}$ assigned color $i$ (indices modulo $n$). Assume that no rainbow
$C_{n-1}$ containing $x_1$ exists.

\smallskip\noindent
\emph{Claim C1: common Hamilton edges.}
Fix $i\in\{2,\ldots,n\}$ and delete the non-root vertex $x_i$ from $C$. The remaining graph is a
rainbow path $P_i$ on $n-1$ vertices, with endpoints $x_{i-1},x_{i+1}$ and missing colors
$i-1,i$; throughout this proof, subscripts are read modulo $n$. The two-edge path
$x_{i-1}x_ix_{i+1}$ is rainbow. If
$\delta(x_{i-1})+\delta(x_{i+1})<n$, condition~\eqref{eq:LO} makes the chord
$x_{i-1}x_{i+1}$ common, and coloring it with either released color closes $P_i$ to the forbidden
rooted $C_{n-1}$. Hence
\[
 \delta(x_{i-1})+\delta(x_{i+1})\ge n.
\]
Apply Lemma~\ref{lem:rooted-path} to $P_i$ with color $i$ at the endpoint $x_{i-1}$ and color
$i-1$ at the endpoint $x_{i+1}$. Since the only vertex outside $P_i$ is $x_i$, the outside
neighborhood bound gives
\[
 d_i(x_{i-1},\{x_i\})+d_{i-1}(x_{i+1},\{x_i\})\ge2.
\]
Both summands are at most one, so
\begin{equation}\label{eq:adjacent-swap}
 x_{i-1}x_i\in E(G_i),\qquad x_ix_{i+1}\in E(G_{i-1}).
\end{equation}
Thus the two colors on the Hamilton edges incident with any non-root vertex may be transposed
without changing the underlying cycle.
Here the exchange is between two distinct colors absent from the path obtained by
deleting that vertex; no color is assigned twice, and the vertex set is unchanged.

Iterate these transpositions, starting from an arbitrary rainbow coloring of the same Hamilton cycle
containing $x_1$.  The preceding argument applies after each transposition, so the two colors on the
Hamilton edges incident with any non-root vertex may again be exchanged. Cutting the
cyclic list of edge positions at the root leaves exactly the adjacent transpositions
$(1,2),(2,3),\ldots,(n-1,n)$; these generate the full symmetric group on the $n$ colors.
Consequently every fixed Hamilton edge can receive every color in some rainbow coloring of the
same rooted Hamilton cycle. Hence every Hamilton edge belongs to every member, that is,
\begin{equation}\label{eq:ham-edges-common}
 x_jx_{j+1}\in E_\cap\qquad(j\in\mathbb Z_n).
\end{equation}
The same argument applies to every rainbow Hamilton cycle containing $x_1$.

\smallskip\noindent
\emph{Claim C2: root-safe distance-two chords are absent.}
If the middle vertex $x_{i+1}$ is not $x_1$, then
\begin{equation}\label{eq:safe-d2-absent}
 x_ix_{i+2}\notin E(G_j)\qquad(j\in[n]).
\end{equation}
Indeed, otherwise delete the two Hamilton edges incident with $x_{i+1}$ and use the chord in
color $j$.  If $j$ is one of the two released colors, this already gives the required rainbow
coloring.  Otherwise the unique Hamilton edge currently carrying color $j$ is common by
\eqref{eq:ham-edges-common}, so recolor it with one of the two released colors.  In either case
deleting $x_{i+1}$ gives a rainbow $C_{n-1}$ containing $x_1$.

\smallskip\noindent
\emph{Claim C3: common distance-three chords.}
Fix $i\in\{2,\ldots,n\}$ and delete $x_i$ from $C$. Orient the remaining path from
$x_{i-1}$ to $x_{i+1}$ in the direction that first visits $x_{i-2}$. It has $n-1$ vertices,
contains $x_1$, and misses colors $i-1,i$. Apply Lemma~\ref{lem:rooted-path} with endpoint
colors $i-1$ at $x_{i-1}$ and $i$ at $x_{i+1}$. The inside-neighborhood sum is at most $n-2$.
The outside vertex $x_i$ contributes one neighbor in each of these two colors. The total-degree
lower bound is the one established in Claim~C1 for the same endpoint pair,
$\delta(x_{i-1})+\delta(x_{i+1})\ge n$. Hence equality holds throughout in the rooted
endpoint-crossing lemma. At index $j=2$ in its equality clause, exactly one of
$x_{i-1}x_{i-3}\in E(G_{i-1})$ and $x_{i-2}x_{i+1}\in E(G_i)$ holds.
For $i\ne3$, the first edge is a safe distance-two chord whose middle vertex $x_{i-2}$
is not the root, so Claim~C2 excludes it and forces
\[
 x_{i-2}x_{i+1}\in E(G_i)\qquad(i\ne3).
\]
Symmetrically, at index $j=k-2$ (where $k=n-1$ is the number of vertices of the path),
exactly one of $x_{i-1}x_{i+2}\in E(G_{i-1})$ and
$x_{i+3}x_{i+1}\in E(G_i)$ holds. For $i\ne n-1$, the latter is a safe distance-two
chord whose middle vertex $x_{i+2}$ is not the root, so Claim~C2 excludes it and forces
\[
 x_{i-1}x_{i+2}\in E(G_{i-1})\qquad(i\ne n-1).
\]
Thus, for $i\in\{2,\ldots,n\}\setminus\{3,n-1\}$, both chords are available. Replace the segment
$x_{i-2}x_{i-1}x_ix_{i+1}x_{i+2}$ by
$x_{i-2}x_{i+1}x_ix_{i-1}x_{i+2}$.  Use colors $i$ and $i-1$ on the two new chords, respectively,
and recolor the two retained middle edges $x_{i+1}x_i$ and $x_ix_{i-1}$ with the released colors
$i+1$ and $i-2$.  These middle edges are common by Claim~C1, and the four colors
$i-2,i-1,i,i+1$ are distinct modulo $n$ for $n\ge6$.
They are exactly the colors released from the original four-edge
segment. For clarity, the exact replacement certificate is
\[
\begin{array}{c|c}
\text{edge in the new segment}&\text{assigned color}\\ \hline
x_{i-2}x_{i+1}&i\\
x_{i+1}x_i&i+1\\
x_ix_{i-1}&i-2\\
x_{i-1}x_{i+2}&i-1
\end{array}
\]
The first and last edges occur in the indicated layers by the
preceding endpoint-crossing deductions; the two middle edges are
common by Claim~C1. The replacement has the same five vertices and four edges,
and every unchanged Hamilton edge keeps its color.  Thus this is another rainbow Hamilton cycle containing $x_1$, so Claim~C1 makes the
two new chords common. Every cyclic distance-three
chord $x_jx_{j+3}$ occurs in one of these switches: reading indices modulo $n$, it is the first
new chord for $i=j+2$ or the second for $i=j+1$.  These two candidate indices are consecutive
modulo $n$.  For $n\ge6$ they cannot both lie in the forbidden set $\{1,3,n-1\}$, and at least
one has a representative in $\{2,\ldots,n\}\setminus\{3,n-1\}$.  In the smallest case
$n=6$, the admissible switch indices $i=2,4,6$ cover all three distinct
undirected distance-three chords (each of the six indexed chords is counted twice).
Hence
\begin{equation}\label{eq:d3-common}
 x_jx_{j+3}\in E_\cap\qquad(j\in\mathbb Z_n).
\end{equation}
Consequently, for any two distinct colors $a,b$ and any $j$, the vertex set
$V\setminus\{x_{j+1},x_{j+2}\}$ supports a rainbow $C_{n-2}$ avoiding colors $a,b$: use the
common chord $x_jx_{j+3}$ and the remaining common Hamilton edges, and color them with
$[n]\setminus\{a,b\}$.

\smallskip\noindent
\emph{Claim C4: a same-vertex two-color degree bound.}
We prove
\begin{equation}\label{eq:same-vertex-sum}
 d_a(v)+d_b(v)\ge n
 \qquad(v\in V,\ a\ne b).
\end{equation}
Otherwise choose a cyclic distance-two vertex $w$ of $v$ so that the middle Hamilton vertex is
not $x_1$. By~\eqref{eq:safe-d2-absent} and~\eqref{eq:LO},
$\delta(v)+\delta(w)\ge n$. Therefore
\[
 [d_a(v)+d_b(v)]+[d_a(w)+d_b(w)]\ge2n,
\]
so $d_a(w)+d_b(w)>n$. Choose a Hamilton neighbor $z$ of $w$ with $z\ne x_1$, and use the common
$C_{n-2}=y_1\cdots y_{n-2}y_1$ from Claim~C3 on $V\setminus\{w,z\}$, colored without $a,b$.
Since $d_a(w)+d_b(w)>n$ and the omitted vertex $z$ contributes at most
one adjacency in each of the two colors, the total number of $a$- and
$b$-neighbors of $w$ on this cycle is at least $n-1$, and hence is greater
than $n-2$. With indices modulo $n-2$, the two subsets
$\{j:wy_j\in E(G_a)\}$ and $\{j:wy_{j+1}\in E(G_b)\}$ of $[n-2]$ intersect.
Inserting $w$ at the corresponding edge produces a rainbow
$C_{n-1}$ whose omitted vertex is $z\ne x_1$, a contradiction.

\smallskip\noindent
\emph{Claim C5: alternation and the extremal structure.}
Fix $v\in V$ and distinct colors $a,b$. Choose a Hamilton neighbor $z$ of $v$ with $z\ne x_1$,
and let
\[
 D=y_1y_2\cdots y_{n-2}y_1
\]
be the common $C_{n-2}$ on $V\setminus\{v,z\}$, colored without $a,b$. Put
\[
 A_j=\mathbf 1_{\{vy_j\in E(G_a)\}},\qquad
 B_j=\mathbf 1_{\{vy_j\in E(G_b)\}}.
\]
Since inserting $v$ in an edge of $D$ would give the forbidden rooted $C_{n-1}$,
\[
 A_j+B_{j+1}\le1,\qquad B_j+A_{j+1}\le1
\]
for every cyclic index $j$. The edge $vz$ is common, and~\eqref{eq:same-vertex-sum} gives
\[
 \sum_jA_j+\sum_jB_j=d_a(v)+d_b(v)-2\ge n-2.
\]
Summing either family of inequalities gives the reverse bound, so equality holds everywhere:
\[
 A_j+B_{j+1}=B_j+A_{j+1}=1,
 \qquad A_{j+2}=A_j.
\]
If $n$ were odd, the cycle $D$ would have odd length and $A_j$ would be constant; one of
$G_a,G_b$ would then give $v$ degree one, impossible in a $2$-connected graph. Hence $n$ is
even. When $n$ is even, the constant patterns are excluded for the same reason, so $A_j$ and
$B_j$ alternate and in fact $A_j=B_j$: from the tight inequalities
$B_{j+1}=1-A_j$ and $A_{j+1}=1-B_j$, the difference
$A_j-B_j$ is independent of $j$, while period two and exclusion of
constant patterns force the common alternating pattern. Thus
\[
 N_{G_a}(v)=N_{G_b}(v),\qquad d_a(v)=d_b(v)=n/2.
\]
As $a,b$ and $v$ were arbitrary, all members coincide with an $n/2$-regular graph $G$.
Finally take $v=x_i$ and choose a Hamilton neighbor $z$ of $v$ which is not the root; such a choice
is always possible because at most one Hamilton neighbor of $v$ is $x_1$.  Traverse $C$ from $v$
towards $z$.  In this orientation the common $C_{n-2}$ on $V\setminus\{v,z\}$ begins with the
vertex at cyclic distance two from $v$, and that distance-two edge is absent by
~\eqref{eq:safe-d2-absent}.  The alternating pattern on the $C_{n-2}$ therefore shows that $v$ is
adjacent precisely to the vertices at odd cyclic distance from $v$.  Since $n$ is even, the parity
of cyclic distance is independent of the chosen orientation. Hence, with
\[
 X=\{x_1,x_3,\ldots,x_{n-1}\},\qquad
 Y=\{x_2,x_4,\ldots,x_n\},
\]
we have $G=K_{X,Y}$. This proves the lemma.
\end{proof}

\section{Cycle spectra: proofs and sharpness}\label{sec:main-proof}

\subsection{The cycle-spectrum theorem}\label{sec:completion}

\begin{lemma}[Triangle obstruction degree bound]\label{lem:boundary}
Suppose $n\ge4$, $\GraphSys$ satisfies~\eqref{eq:LO}, and $\GraphSys$ contains a rainbow
Hamilton cycle but no rainbow triangle. Then
\[
 \sum_{v\in V}\delta(v)\ge \frac{n^2}{2}.
\]
\end{lemma}

\begin{proof}
Fix a rainbow Hamilton cycle $C=v_1\cdots v_nv_1$. If
$\delta(v_i)+\delta(v_{i+2})<n$, then~\eqref{eq:LO} makes $v_iv_{i+2}$ common, and this
chord together with $v_iv_{i+1}v_{i+2}$ forms a rainbow triangle: assign the chord any color
distinct from the two colors on the Hamilton edges. Thus
$\delta(v_i)+\delta(v_{i+2})\ge n$ for every $i$. Summing gives
\[
 2\sum_{v\in V}\delta(v)
 =\sum_{i=1}^n\bigl(\delta(v_i)+\delta(v_{i+2})\bigr)\ge n^2.
\]
\end{proof}

\begin{proof}[Proof of Theorem~\ref{thm:main}]
The rooted lengths are supplied by the internal Lemmas~\ref{lem:rooted-c4},
\ref{lem:rooted-cn1} and Theorem~\ref{thm:rooted-rigidity}, together
with Hamiltonicity from~\ref{thm:ham-input}. The aggregate-degree
input~\ref{thm:bondy-input} is used only in the final unrooted
triangle step.
Assume first that $\GraphSys$ is not the common balanced complete bipartite system. Fix an
arbitrary root $r\in V$. Theorem~\ref{thm:ham-input} gives an $r$-rooted rainbow $C_n$, while
Lemma~\ref{lem:rooted-c4} gives an $r$-rooted rainbow $C_4$.

If $n=4$, this already gives the vertex-specific conclusion. If $n=5$, the Hamilton cycle supplies
$C_5$. If $n=6$, Lemma~\ref{lem:rooted-cn1} supplies an $r$-rooted $C_5$, and the Hamilton cycle
supplies $C_6$.  For $n=7$, Theorem~\ref{thm:rooted-rigidity} is used only at $m=4$, producing
an $r$-rooted $C_5$; Lemma~\ref{lem:rooted-cn1} and the Hamilton cycle then supply lengths $6$
and $7$.  Suppose $n\ge8$. Starting from the rooted $C_4$, apply
Theorem~\ref{thm:rooted-rigidity} successively: because the exceptional outcome is excluded, an
$r$-rooted $C_m$ with $4\le m\le n-3$ forces an $r$-rooted $C_{m+1}$. Hence $r$ lies on rainbow
cycles of every length from $4$ through $n-2$. Lemma~\ref{lem:rooted-cn1} gives length $n-1$,
and Theorem~\ref{thm:ham-input} gives length $n$. Since $r$ was arbitrary, $\GraphSys$ is
$[4,n]$-rainbow vertex-pancyclic.

It remains only to obtain a rainbow triangle. If none existed, Lemma~\ref{lem:boundary} would give
$\sum_v\delta(v)\ge n^2/2$; Theorem~\ref{thm:bondy-input} would then imply that $\GraphSys$ is
rainbow pancyclic or is the exceptional balanced bipartite system. The latter is excluded, and the
former contains a rainbow triangle, a contradiction. Thus $\GraphSys$ is rainbow pancyclic as
well. This proves the first outcome of the theorem; the remaining possibility is exactly the
second outcome.
\end{proof}

\subsection{The stronger local cutoff and rooted triangles}\label{sec:local-43-proof}

\begin{proof}[Proof of Theorem~\ref{thm:local-43}]
For $n=3$, every $2$-connected member is $K_3$, so the conclusion is immediate. Assume henceforth
that $n\ge4$ and write
\[
 \tau=\frac{4n}{3}-1.
\]
Since $\tau>n$, condition~\eqref{eq:LO-43} implies~\eqref{eq:LO}. Hence
Theorem~\ref{thm:main} applies. Its exceptional balanced complete bipartite outcome cannot occur:
if $u,v$ belong to the same part, then $d^c(u,v)=2$, $\delta(u)+\delta(v)=n<\tau$, but
$uv\notin\Ec$. Consequently, every vertex already lies on a rainbow $C_\ell$ for each
$4\le\ell\le n$. It remains to force a rainbow triangle through an arbitrary prescribed vertex
$x$.

Assume that no rainbow triangle contains $x$.  We first show that every edge incident with $x$
which occurs in one member is common.  Suppose
$xy\in E(G_i)$ but $xy\notin\Ec$. Since $d^c(x,y)=1$, the contrapositive of
\eqref{eq:LO-43} gives
\[
 \delta(x)+\delta(y)\ge\tau>n.
\]
Since $n\ge4$, choose distinct colors $j,k\in[n]\setminus\{i\}$. Then
\[
 d_j(x)+d_k(y)\ge\delta(x)+\delta(y)\ge\tau>n,
\]
so $N_j(x)\cap N_k(y)\ne\varnothing$. Any vertex
$z$ in this intersection is distinct from $x,y$, and $xyzx$ is a rainbow triangle in the three
pairwise distinct colors $i,k,j$, a contradiction. Thus all edges incident with $x$ are common.

It follows that the neighborhood of $x$ is the same in every member; write
\[
 N=N_{G_i}(x)\quad(i\in[n]),\qquad d=|N|=\delta(x).
\]
Moreover, $N$ is independent in every $G_i$. Otherwise, if $u,v\in N$ and
$uv\in E(G_k)$, then the two common edges $xu,xv$ can be assigned two distinct colors different
from $k$, producing a rainbow triangle through $x$.

Since every $G_i$ is $2$-connected, $d\ge2$. For distinct $u,v\in N$, the path $uxv$ can be
colored rainbow, while $uv\notin\Ec$. Hence the contrapositive of~\eqref{eq:LO-43} gives
\[
 \delta(u)+\delta(v)\ge\tau.
\]
Because $N$ is independent in every member, a vertex of $N$ can be adjacent only to $x$ and
to the $n-d-1$ vertices of $V\setminus(N\cup\{x\})$. Hence each of $u,v$ has degree at most
$n-d$ in every member. Therefore
\[
 \tau\le2(n-d),
 \qquad\text{so}\qquad
 d\le \frac n3+\frac12.
 \tag{\(*\)}\label{eq:root-degree-upper-43}
\]
For $n=4$, this contradicts $d\ge2$, so assume $n\ge5$. Then
$S=V\setminus(N\cup\{x\})$ is nonempty.

Fix a color $i$. Since $G_i$ is $2$-connected, $G_i-x$ is connected. The sets $N$ and $S$
are both nonempty and partition $V\setminus\{x\}$, so some edge $uy\in E(G_i)$ crosses from
$u\in N$ to $y\in S$. The edge $xu$ is common, so assigning it a color different from $i$
makes $xuy$ a rainbow $2$-path. Since every edge incident with $x$ that occurs in one member is
common, $y\notin N$ implies that $xy$ belongs to no member. The contrapositive of
\eqref{eq:LO-43} therefore gives
\[
 d+\delta(y)\ge\tau.
\]
Using~\eqref{eq:root-degree-upper-43},
\[
 \delta(y)\ge \tau-d
 \ge \left(\frac{4n}{3}-1\right)-\left(\frac n3+\frac12\right)
 =n-\frac32.
\]
As $\delta(y)$ is an integer, $\delta(y)\ge n-1$. Hence $y$ is adjacent to $x$ in every
$G_j$, contradicting $y\notin N$. This contradiction proves that every prescribed vertex $x$
lies on a rainbow triangle. Together with the already established rooted cycle lengths
$4,\ldots,n$, the system is rainbow vertex-pancyclic.
\end{proof}

\begin{remark}
Only the bound $d\le n/3+1/2$ in~\eqref{eq:root-degree-upper-43} is needed in the last step.
Thus the triangle argument is genuinely local: after Theorem~\ref{thm:main} supplies the rooted
cycle lengths $4,\ldots,n$, no global Ore degree-sum hypothesis is used.
\end{remark}

\subsection{Sharpness and examples}\label{sec:examples}

The exceptional system in Theorem~\ref{thm:main} is rainbow bipancyclic.
If all members equal $K_{t,t}$ with $t\ge2$, then every vertex has
system degree $t=n/2$, so \eqref{eq:LO} holds. There are no odd
cycles. For every $2\le j\le t$, choose $j$ vertices from each part
and alternate between the parts; assigning any $2j$ distinct colors
gives a rainbow $C_{2j}$. Thus the exception is rainbow bipancyclic.

\begin{proposition}[The lower endpoint $4$ is best possible]\label{prop:vertex-sharp}
For every even $n=2t\ge6$, there is a $2$-connected diagonal system satisfying~\eqref{eq:LO}
which is not the exceptional balanced complete bipartite system and is not rainbow
vertex-pancyclic.
\end{proposition}

\begin{proof}
Let $V=X\sqcupdot Y$ with $|X|=|Y|=t$, choose distinct $y_1,y_2\in Y$, and let
\[
 G=K_{X,Y}+y_1y_2,
 \qquad G_1=\cdots=G_n=G.
\]
The graph $G$ is $2$-connected. Every vertex has degree at least $t=n/2$, so the antecedent
$\delta(u)+\delta(v)<n$ in~\eqref{eq:LO} never occurs. The system is not the exceptional one
because of the extra edge $y_1y_2$. However, for any
$y_0\in Y\setminus\{y_1,y_2\}$ we have $N_G(y_0)=X$, and $X$ is independent. Thus no triangle
contains $y_0$, so $y_0$ lies on no rainbow triangle in the diagonal system.
\end{proof}

For every order $n\equiv1\pmod3$, the stronger local cutoff is sharp at the integer level.
We record the diagonal construction underlying the corresponding example of Li, Wang and
Yan~\cite{LiWangYan}, because it also verifies the local condition directly.

\begin{proposition}[Sharpness of the local $4n/3$ cutoff]\label{prop:local-43-sharp}
For every $n\equiv1\pmod3$ with $n\ge4$, there is a $2$-connected diagonal system which
satisfies the local condition obtained from \eqref{eq:LO-43} by replacing the cutoff
$4n/3-1$ with $(4n-4)/3$, but which is not rainbow vertex-pancyclic.
\end{proposition}

\begin{proof}
Set
\[
 t=\frac{n+2}{3},\qquad m=\frac{2n-5}{3},
\]
and let
\[
 G=\overline{K_t}\vee(K_1\cup K_m).
\]
Write $Y$ for the independent set of order $t$, let $x$ be the vertex of the $K_1$, and let
$Z$ induce the clique $K_m$.  Thus every vertex of $Y$ is adjacent to $x$ and to every vertex of
$Z$, while there are no edges from $x$ to $Z$.  Take $G_1=\cdots=G_n=G$.

Since $t\ge2$ and $m\ge1$, deleting any one vertex leaves a connected graph.  If $x$ is
deleted, every vertex of $Y$ is adjacent to every vertex of $Z$; if a vertex of $Y$ is deleted,
another vertex of $Y$ remains and links $x$ to $Z$; and if a vertex of $Z$ is deleted, $x$ remains
adjacent to all of $Y$, while every surviving vertex of $Z$ is adjacent to $Y$.  Hence $G$ is
$2$-connected.  The degrees are
\[
 d_G(x)=t,\qquad d_G(y)=m+1\ (y\in Y),\qquad d_G(z)=n-2\ (z\in Z).
\]
The only nonedges are the pairs $xz$ with $z\in Z$ and the pairs inside $Y$.  For each such
pair the degree sum is exactly
\[
 t+(n-2)=2(m+1)=\frac{4n-4}{3}.
\]
Moreover every such pair has distance two.  Therefore the strict local condition with cutoff
$(4n-4)/3$ holds: no nonedge of rainbow distance at most two has degree sum strictly below the
cutoff.  On the other hand
\[
 N_G(x)=Y
\]
is independent, so $x$ is contained in no triangle.  The diagonal system is therefore not
rainbow vertex-pancyclic.
\end{proof}

Thus the strict cutoff in Theorem~\ref{thm:local-43} cannot in general be decreased to
$(4n-4)/3$.  Equivalently, in the integer formulation above, the theorem requires exactly the
next degree-sum level beyond this construction.

The following diagonal construction from~\cite{ChenChen} shows that
~\eqref{eq:LO} is strictly weaker than~\eqref{eq:global-ore}.

\begin{proposition}[The localization is strict]\label{prop:strict}
For every $n\ge6$, there is a system of $2$-connected graphs satisfying
\eqref{eq:LO} but not \eqref{eq:global-ore}.
\end{proposition}

\begin{proof}
Write $W=A\sqcupdot B$, where $|W|=n-2$ and $|A|,|B|\ge2$.
Let $W$ induce a clique, add two nonadjacent vertices $x,y$, and put
$N(x)=A$, $N(y)=B$. Call the resulting graph $G$, and take
$G_i=G$ for every $i$.
The clique remains connected after any one vertex is removed, and
each remaining vertex among $x,y$ retains a neighbor in it. Thus $G$
is $2$-connected. The pair $x,y$ has distance three and degree sum
$n-2<n$, so \eqref{eq:global-ore} fails.
All other nonedges are of the form $xb$, $b\in B$, or $ya$,
$a\in A$. Their degree sums are, respectively,
$|A|+n-2\ge n$ and $|B|+n-2\ge n$.
Therefore \eqref{eq:LO} holds. In this diagonal system rainbow
distance is ordinary distance, since any simple path can be assigned
distinct colors.
\end{proof}

\begin{proposition}[Strict localization at the stronger cutoff]\label{prop:strict-43}
For every even $n\ge12$ and every odd $n\ge15$, there is a
$2$-connected diagonal graph system satisfying
\eqref{eq:LO-43} which fails the corresponding global Ore
condition
\[
 uv\notin E_\cap\quad\Longrightarrow\quad
 \delta(u)+\delta(v)\ge \frac{4n}{3}-1.
\]
\end{proposition}

\begin{proof}
Use the graph $G$ of Proposition~\ref{prop:strict} with
$|A|=\lfloor(n-2)/2\rfloor$ and
$|B|=\lceil(n-2)/2\rceil$. The only nonedges are
$xy$, the pairs $xb$ with $b\in B$, and the pairs
$ya$ with $a\in A$. We have $d_G(x,y)=3$ and
$d_G(x)+d_G(y)=n-2$, so the stated global condition fails.
The two latter types of nonedges have ordinary distance two
and degree sums at least
\[
 (n-2)+\left\lfloor\frac{n-2}{2}\right\rfloor.
\]
For even $n$ this lower bound is $3(n-2)/2$, which is at
least $4n/3-1$ when $n\ge12$. For odd $n$ it is
$(3n-7)/2$, which is at least $4n/3-1$ when $n\ge15$.
Thus every nonedge at rainbow distance at most two has
system-degree sum at least $4n/3-1$.
The graph is $2$-connected by Proposition~\ref{prop:strict},
and the diagonal system satisfies \eqref{eq:LO-43}.
\end{proof}

\begin{proposition}[The degree-sum cutoff]\label{prop:sharp}
The constant $n$ in the strict inequality $\delta(u)+\delta(v)<n$ appearing in
\eqref{eq:LO} cannot in general be lowered to $n-1$, even when every member is
$2$-connected.
\end{proposition}

\begin{proof}
Let $n=2t+1$ with $t\ge2$, and take every member to be
$K_{t,t+1}$ on the same bipartition. All members are $2$-connected
and every system degree is at least $t$. Thus no pair has degree
sum less than $2t=n-1$, and the condition with threshold $n-1$
holds. Nevertheless, the system has no odd cycle, in particular no
Hamilton cycle, and it is not the even-order exception of
Theorem~\ref{thm:main}.
\end{proof}

Theorem~\ref{thm:main} therefore localizes the global $[4,n]$-vertex-pancyclic Ore theorem of
Li, Wang and Yan~\cite{LiWangYan}, while Proposition~\ref{prop:vertex-sharp} shows why the
vertex-specific conclusion cannot in general be extended to length $3$ under~\eqref{eq:LO}.
Theorem~\ref{thm:local-43} gives the corresponding local analogue of their stronger full
vertex-pancyclicity theorem at the cutoff $4n/3-1$; Proposition~\ref{prop:strict-43}
shows that the stronger localization is also strict for every sufficiently large order.

\section*{Acknowledgments}
This work was supported by the National Natural Science Foundation of China (Grant No.~12171088).

\noindent\textbf{Use of AI-assisted tools.} ChatGPT (OpenAI, GPT-6; October 2026) was used to assist with literature and reference checks, to suggest and scrutinize proof arguments and auxiliary computational checks, and to revise parts of the exposition. These outputs were used as advisory material. The authors retain responsibility for the mathematical arguments, bibliographic accuracy, and final manuscript.

% Before submission, every author should check the accuracy of the
% overlap/dependency statement in the cover letter and the declared tool use.
\section*{Declarations}
\noindent\textbf{Data availability.} This is a theoretical study; no research dataset underlies the results.

% Author action before submission: agree and disclose any conflicts of interest
% in the Wiley portal. Do not claim there are none without all authors' approval.

\end{document}